\documentclass{article}

\usepackage{amsmath}
\usepackage{amsfonts}
\usepackage{amsthm}
\usepackage{amssymb}
\usepackage[all,cmtip]{xy}
\usepackage[bookmarks=true]{hyperref}
\usepackage{enumerate}

\newtheorem{thm}{Theorem}[section]
\newtheorem*{thm*}{Theorem}
\newtheorem{introthm}{Theorem} % theorem style for intro

\newtheorem{prop}[thm]{Proposition}
\newtheorem{cor}[thm]{Corollary}

\newtheorem{lemma}[thm]{Lemma}
\theoremstyle{definition}
\newtheorem{defn}[thm]{Definition}
\newtheorem*{acknowledgements*}{Acknowledgements}
\newtheorem{remark}[thm]{Remark}

\newcommand{\F}{\mathbb{F}}
\newcommand{\Q}{\mathbb{Q}}
\newcommand{\Z}{\mathbb{Z}}
\newcommand{\R}{\mathbb{R}}
\newcommand{\proj}{\mathbb{P}}
\newcommand{\aff}{\mathbb{A}}
\newcommand{\set}[1]{\left\{#1\right\}}
\newcommand{\eps}{\varepsilon}
\newcommand{\inject}{\hookrightarrow}

\DeclareMathOperator{\Sing}{Sing} % singular locus
\DeclareMathOperator{\chr}{char} % characteristic of a field
\DeclareMathOperator{\bigO}{O} % big-O notation
\DeclareMathOperator{\codim}{codim} % codimension
\newcommand{\Frob}{\mathrm{Frob}} % Frobenius
\newcommand{\id}{\mathrm{id}} % identity
\DeclareMathOperator{\Tr}{Tr} % trace
\DeclareMathOperator{\Gal}{Gal} % Galois group
\DeclareMathOperator{\Var}{Var} % variance
\DeclareMathOperator*{\Covar}{Covar} % covariance
\DeclareMathOperator{\Gr}{Gr}

\numberwithin{equation}{section}

\allowdisplaybreaks[1]

\title{Higher moments of arithmetic functions in short intervals: a geometric perspective}
\author{Daniel Hast, Vlad Matei}

\begin{document}
\maketitle

\begin{abstract}
  We study the geometry associated to the distribution of certain arithmetic functions, including the von Mangoldt function and the M\"obius function, in short intervals of polynomials over a finite field $\F_q$. Using the Grothendieck--Lefschetz trace formula, we reinterpret each moment of these distributions as a point-counting problem on a highly singular complete intersection variety. We compute part of the $\ell$-adic cohomology of these varieties, corresponding to an asymptotic bound on each moment for fixed degree $n$ in the limit as $q \to \infty$. The results of this paper can be viewed as a geometric explanation for asymptotic results that can be proved using analytic number theory over function fields.
\end{abstract}

Fix integers $m \geq 1$, $n \geq 4$, and $1 \leq h \leq n - 3$. Let $q = p^r$ be a power of a prime $p$. Denote the $q$-power (geometric) Frobenius endomorphism by $\Frob_q$. Let $M_n(\F_q) \subset \F_q[x]$ be the set of monic polynomials of degree $n$. Define an equivalence relation $\sim$ (or $\sim_h$ if we wish to make the dependence on $h$ explicit) by
\[
f \sim g \iff \deg(f - g) \leq h.
\]
The equivalence classes of $\sim_h$ are the function field analogue of intervals of width $q^{h + 1}$.

For a suitable arithmetic function $\varphi$ on $\F_q[x]$, we will show that certain asymptotic bounds on
\[
\frac{1}{q^n} \sum_{f \in M_n(\F_q)} \Bigl( \sum_{\substack{g \in \F_q[x] \\ \deg g \leq h}} \varphi(f + g) \Bigr)^m = q^{h + 1} \cdot \frac{1}{q^n} \sum_{\substack{f_1, \dots, f_m \in M_n(\F_q) \\ f_i \sim f_j}} \varphi(f_1) \cdots \varphi(f_m)
\]
in the limit as $q \to \infty$ (ranging over powers of primes $p > n$) can be interpreted in terms of vanishing of cohomology of a certain algebraic variety. These sums are the moments of the distribution of $\varphi$ among intervals of size $q^{h + 1}$ in degree $n$.

We begin with a brief summary in \S\ref{section:intro} of the integer analogues of the questions we study. We also discuss the function field analogue of the prime number theorem, as well as recent results of a similar nature, stating our arithmetic bounds and placing them into context.

In \S\ref{section:geometry}, we define an algebraic variety $X_{m, n, h}$ that parametrizes the roots of $m$-tuples of polynomials $(f_1, \dots, f_m)$ such that $f_i \sim f_j$ for all $i, j$. We compute that $X_{m, n, h}$ is an irreducible projective complete intersection of dimension $n + (m - 1)(h + 1) - 1$ with singular locus of codimension $2h + 3$.

In \S\ref{section:cohomology}, we study the $\ell$-adic cohomology of $X_{m, n, h}$. The variety $X_{m, n, h}$ carries a Frobenius action and an action of $S_n^m$ given by permuting the roots. A general theorem on the cohomology of projective complete intersections \cite[Prop.\ 3.3]{sing-var} shows that $H^r(X_{m, n, h}, \Q_\ell)$ agrees with the cohomology of projective space in high degrees. This shows that the trace of Frobenius on $H^r(X_{m, n, h}, \Q_\ell)$ for such $r$ is $q^{r/2}$ for $r$ even, and $0$ for $r$ odd. We also show that $S_n^m$ acts trivially on $H^r(X_{m, n, h}, \Q_\ell)$ in these cases. These results are summarized as follows:
\begin{introthm}
  \label{thm:intro}
  Let $m \geq 2$, $n \geq 4$, and $1 \leq h \leq n - 3$. Let $k$ be an algebraically closed field of characteristic $p$; if $m > 2$, assume $p > n$ or $p = 0$. The closed subvariety $X_{m, n, h} \subset \proj_k^{mn - 1}$, defined in \S\ref{section:def-variety}, is an irreducible complete intersection of dimension $n + (m - 1)(h + 1) - 1$. The singular locus of $X_{m, n, h}$ has codimension $2h + 3$.

  Moreover, for each $r$ such that $2 \dim X_{m, n, h} - 2h - 1 \leq r \leq 2 \dim X_{m, n, h}$ and for any prime $\ell \neq p$,
  \[
  H^r(X_{m, n, h}, \Q_\ell) = \begin{cases}
    \Q_\ell(-r/2) & \text{if $r$ is even}, \\
    0 & \text{if $r$ is odd}.
  \end{cases}
  \]
  Finally, letting $S_n^m = S_n \times \dots \times S_n$ act on $X_{m, n, h}$ by the $i$-th copy of $S_n$ permuting the roots of the $i$-th polynomial $f_i$, the induced action of $S_n^m$ on $H^r(X_{m, n, h}, \Q_\ell)$ is trivial.
\end{introthm}

In \S\ref{section:arithmetic}, we discuss the arithmetic bounds that are implied by our geometric results. We apply the Grothendieck--Lefschetz trace formula to the results of the previous section, yielding an asymptotic formula in the large $q$ limit for the number of points of $X_{m, n, h}$ such that the action of $\Frob_q$ induces a specified permutation of the roots. The moments of several arithmetic functions, most notably the von Mangoldt function $\Lambda$ and the M\"obius function $\mu$, are interpreted in this way. In those two cases, this yields the following bounds:
\begin{introthm}
  \label{thm:intro-arith-top}
  For any integers $m \geq 2$, $n \geq 4$, and $1 \leq h \leq n - 3$, there are constants $C_{m, n, h}$ and $D_{m, n, h}$ such that for every prime $p$ (assuming $p > n$ if $m > 2$) and every positive integer power $q = p^r$,
  \[
  \bigg\lvert \frac{1}{q^n} \sum_{f \in M_n(\F_q)} \Bigl( \sum_{\substack{g \in \F_q[x] \\ \deg g \leq h}} [\Lambda(f + g) - 1] \Bigr)^m \bigg\rvert \leq C_{m, n, h} q^{(h + 1)(m - 1)} \\
  \]
  and
  \[
  \bigg\lvert \frac{1}{q^n} \sum_{f \in M_n(\F_q)} \Bigl( \sum_{\substack{g \in \F_q[x] \\ \deg g \leq h}} \mu(f + g) \Bigr)^m \bigg\rvert \leq D_{m, n, h} q^{(h + 1)(m - 1)}.
  \]
\end{introthm}

The upper bounds in Theorem \ref{thm:intro-arith-top} are not new: they are proved for the von Mangoldt and M\"obius functions by Keating and Rudnick \cite{KR}\cite{KR-mobius} and for more general arithmetic functions $\alpha$ by Rodgers \cite{Rodgers}. In fact, they prove more, not only showing that the bound holds for $m = 2$ (which implies the bound for all $m$), but giving an asymptotic of the form 
\[
\Var_{f \in M_n(\F_q)} \sum_{f \sim_h g} \alpha(g) = C_\alpha q^{h + 1} + O(q^{h + 1/2})
\]
for explicit constants $C_\alpha$. Their arguments use powerful equidistribution results of Katz and Sarnak; one of the motivations behind the present work is to show that the order of growth of the short-interval moments in the large $q$ limit can be obtained by more elementary algebraic geometry.  

Furthermore, in \S\ref{section:variance-cohom}, the combination of Rodgers' result and Theorem \ref{thm:intro} allows us to explicitly describe the top weight piece of the first ``interesting'' cohomology group of $X_{2, n, h}$ as a representation of $S_n \times S_n$. Often in arithmetic statistics, one uses facts about the geometry of moduli spaces to derive conclusions about counting problems over function fields; this result adds to the much shorter list (e.g., Browning--Vishe \cite{BV}) of situations where an argument in analytic number theory is used to prove something about the geometry of a moduli space. The resulting theorem gives a geometric interpretation of Rodgers' main theorem:
\begin{introthm}
  \label{thm:intro2}
  Fix integers $n$ and $h$ with $1 \leq h \leq n - 5$. Let
  \[
  \mathcal{H} := \Gr_W^{2n - 2} H^{2n - 2}(X_{2, n, h}, \Q_\ell).
  \]
  Then $\Frob_q$ acts on $\mathcal{H}$ by scalar multiplication by $q^{n - 1}$, and the action of $S_n \times S_n$ on $X_{2, n, h}$ induces an action on $\mathcal{H}$ such that
  \[
  \mathcal{H} \cong \bigoplus_{\substack{\lambda \vdash n \\ \lambda_1 \leq n - h - 2}} V_\lambda \boxtimes V_\lambda
  \]
  as an $S_n \times S_n$-representation. (The sum is over all partitions $\lambda$ of $n$ with all cycles of length at most $n - h - 2$, and $V_\lambda$ is the irreducible representation of $S_n$ associated to $\lambda$.)
\end{introthm}

It would be very interesting to know whether Theorem \ref{thm:intro2} can be derived by purely algebro-geometric means.

\section{Analytic background}
\label{section:intro}
\subsection{Primes in short intervals}
The starting point is the classical prime number theorem, which states that
\[
\pi(X) \sim \frac{X}{\log(X)}
\]
as $X \to \infty$, where $\pi(X)$ is the number of prime numbers $\leq X$. Defining the von Mangoldt function by
\[
\Lambda(n) = \begin{cases}
  \log(p) & \text{if } n = p^k,\ p \text{ prime}, \\
  0 & \text{otherwise},
\end{cases}
\]
the prime number theorem is equivalent to the asymptotic formula
\[
\sum_{n \leq X} \Lambda(n) \sim X.
\]

One may ask what happens if we only look at an interval of width $H$ centered on $X$. If $H$ grows at least as fast as $X^{\frac{1}{2} + \eps}$ for some $\eps > 0$, then the Riemann hypothesis implies
\[
\psi(X; H) := \sum_{X - \frac{H}{2} \leq n \leq X + \frac{H}{2}} \Lambda(n) \sim H.
\]
Of course, as $H$ grows more slowly relative to $X$, the behavior of $\psi(X; H)$ becomes less regular. (In the extreme case of $H$ constant, this is the subject of the Hardy--Littlewood prime tuples conjecture.)

Goldston and Montgomery \cite{GM} studied the variance of $\psi(X; H)$ while allowing for shorter intervals, namely, where $X^\delta < H < X^{1 - \delta}$ for some $\delta > 0$. They proved that, assuming the Riemann hypothesis, the estimate
\[
\frac{1}{X} \int_{2}^{X} (\psi(x; H) - H)^2 dx \sim H \log(X/H)
\]
is equivalent to a strong form of Montgomery's pair correlation conjecture on the distribution of nontrivial zeroes of the Riemann zeta function.

\subsection{The prime polynomial theorem}
Under the function field analogy, $\Z$ corresponds to $\F_q[x]$, prime numbers correspond to irreducible polynomials, ``positive'' corresponds to ``monic'', and $\log$ corresponds to degree. Thus, we define the von Mangoldt function for $f \in \F_q[x]$ by
\[
\Lambda(f) = \begin{cases}
  \deg(g) & \text{if } f = g^k,\ g \text{ irreducible}, \\
  0 & \text{otherwise}.
\end{cases}
\]

The ``prime polynomial theorem'' for $\F_q[x]$ is the statement that
\[
\sum_{f \in M_n(\F_q)} \Lambda(f) = q^n.
\]
As discussed in \cite[\S3.1]{KR}, this follows easily from the fact that the zeta function of $\aff_{\F_q}^1$ has the very simple form
\[
Z(T) = \frac{1}{1 - qT}.
\]

We now give a different proof that reflects our geometric viewpoint.

\begin{proof}
  Let $f \in M_n(\F_q)$, and let $z_1, \dots, z_n \in \F_{q^n}$ be the roots of $f$ (counted with multiplicity). Let $\Frob_q: z \mapsto z^q$ be the Frobenius endomorphism, which generates the cyclic group $\Gal(\F_{q^n}/\F_q)$. Fix an $n$-cycle $\sigma \in S_n$. Then $f = g^k$ for some irreducible $g$ if and only if $\Frob_q(z_i) = z_{\sigma(i)}$ for all $i$ (possibly after reordering the $z_i$), in which case $\Lambda(f)$ is the number of distinct roots of $f$.

  Conversely, for each $z \in \F_{q^n}$, we have $\prod_{i=0}^{n - 1} (x - z^{q^i}) \in M_n(\F_q)$; given $f \in M_n(\F_q)$, there are exactly $\Lambda(f)$ ways $f$ can arise in this way. So
  \[
  \sum_{f \in M_n(\F_q)} \Lambda(f) = \# \F_{q^n} = q^n.
  \qedhere
  \]
\end{proof}

Proposition \ref{prop:m=1} uses essentially the same argument (for a more general class of arithmetic functions), reformulated in terms of $\ell$-adic \'etale cohomology via the Grothendieck--Lefschetz trace formula. The fact that certain arithmetic functions can be expressed in terms of the action of Frobenius is central to our method.

\subsection{Prime polynomials in short intervals}
\label{section:short-intervals}
As mentioned earlier, ``nearness'' in $M_n(\F_q)$ is given by the condition $\deg(f - g) \leq h$, in which case we write $f \sim g$. This relation partitions $M_n(\F_q)$ into intervals of size $H = q^{h + 1}$. Given an arithmetic function $\varphi$, we are interested in the distribution of
\[
\psi_\varphi(g; h) := \sum_{f \sim g} \varphi(f)
\]
as $g$ ranges over $M_n(\F_q)$.

In addition to the asymptotics as $n \to \infty$, we may also study what happens as $q \to \infty$. So far, the latter limit has proved more tractable.

For $\varphi = \Lambda$, Keating and Rudnick \cite{KR} studied the variance of $\psi_\Lambda(g; h)$ for fixed $n$ in the limit as $q \to \infty$, ranging over all prime powers $q$. Using an equidistribution theorem of Katz, they proved:

\begin{thm}[{\cite[Thm.\ 2.1]{KR}}]
  Let $1 \leq h < n - 3$. Then
  \[
  \lim_{q \to \infty} \frac{1}{q^{h + 1}} \Var \psi_\Lambda(-; h) = n - h - 2.
  \]
\end{thm}

One may similarly define the M\"obius function $\mu$ on $\F_q[x]$ and study the distribution of $\mu$ in short intervals. Using similar methods, Keating and Rudnick prove an analogous result in a recent paper:

\begin{thm}[{\cite[Thm.\ 1.2]{KR-mobius}}]
  Let $0 \leq h \leq n - 5$. Then as $q \to \infty$ with $q$ odd,
  \[
  \Var \psi_\mu(-, h) \sim q^{h + 1}.
  \]
\end{thm}

Rodgers generalized this to a larger class of arithmetic functions: We say $\alpha: M_n(\F_q) \to \mathbb{C}$ is a \emph{factorization function} if $\alpha(f)$ depends only on the degrees and multiplicities of the factorization of $f$. We can uniquely decompose $\alpha$ as a sum
\[
\alpha(f) = \sum_{\lambda \vdash n} \hat{\alpha}_\lambda X^\lambda(f) + b(f),
\]
where $b$ is supported on the non-squarefree locus of $M_n(\F_q)$, and $X^\lambda$ is the irreducible character of $S_n$ associated to a partition $\lambda$ of $n$. For $f \in M_n(\F_q)$, we define $X^\lambda(f) = 0$ if $f$ is not squarefree, and $X^\lambda(f) = X^\lambda(\sigma_f)$ if $f$ is squarefree and $\sigma_f \in S_n$ has the same cycle type as the action of Frobenius $\Frob_q$ on the roots of $f$.

\begin{thm}[{\cite[Thm.\ 3.1]{Rodgers}}]
  \label{thm:rodgers}
  Fix integers $h$ and $n$ with $0 \leq h \leq n - 5$. Let $\alpha: M_n(\F_q) \to \mathbb{C}$ be a factorization function as above. Then
  \[
  \Var_{f \in M_n(\F_q)} \Bigl( \sum_{f \sim_h g} \alpha(g) \Bigr) = q^{h + 1} \sum_{\substack{\lambda \vdash n \\ \lambda_1 \leq n - h - 2}} \lvert \hat{\alpha}_\lambda \rvert^2 + \bigO(q^{h + 1/2}),
  \]
  where $\lambda = (\lambda_1, \lambda_2, \dots, \lambda_k)$ ranges over all partitions of $n$ (with $\lambda_1 \geq \lambda_2 \geq \dots \geq \lambda_k$ by convention) such that $\lambda_1 \leq n - h - 2$.
\end{thm}

We will use the above bounds in \S\ref{section:variance-cohom} to prove Theorem \ref{thm:intro2}.

In \S\ref{section:arithmetic}, we will use the geometric results stated above in Theorem \ref{thm:intro} to deduce the following bounds (proved in Corollaries \ref{cor:von-mangoldt} and \ref{cor:mobius}):
\begin{thm}
  \label{thm:intro-arith}
  For any integers $m \geq 2$, $n \geq 4$, and $1 \leq h \leq n - 3$, there are constants $C_{m, n, h}$ and $D_{m, n, h}$ such that for every prime $p$ (assuming $p > n$ if $m > 2$) and every positive integer power $q = p^r$,
  \[
  \bigg\lvert \frac{1}{q^n} \sum_{f \in M_n(\F_q)} \Bigl( \sum_{\substack{g \in \F_q[x] \\ \deg g \leq h}} [\Lambda(f + g) - 1] \Bigr)^m \bigg\rvert \leq C_{m, n, h} q^{(h + 1)(m - 1)} \\
  \]
  and
  \[
  \bigg\lvert \frac{1}{q^n} \sum_{f \in M_n(\F_q)} \Bigl( \sum_{\substack{g \in \F_q[x] \\ \deg g \leq h}} \mu(f + g) \Bigr)^m \bigg\rvert \leq D_{m, n, h} q^{(h + 1)(m - 1)}.
  \]
\end{thm}

The new contribution here is that these bounds have a purely geometric origin; conversely, Theorem \ref{thm:intro} does not follow from the above arithmetic bounds. For $m = 2$, this recovers the above theorems of Keating and Rudnick, except their explicit constants are replaced with an $\bigO(1)$. We also prove similar bounds for a larger class of arithmetic functions that includes many factorization functions in the sense of \cite[2.2]{Rodgers}. Note also the similarity of the formula for $\mu$ to the version of the Chowla conjecture for $\F_q[x]$ proved by Carmon and Rudnick:
\begin{thm*}[{\cite[Thm.\ 1.1]{CR}}]
  Let $\F_q$ be a finite field of odd characteristic. Fix $m > 1$ and $n > 1$. Let $a_2, \dots, a_m$ be distinct polynomials in $\F_q[x]$ of degree less than $n$. Then for any $e_1, \dots, e_m \in \set{1, 2}$, not all even,
  \[
  \bigg\lvert \frac{1}{q^n} \sum_{f \in M_n(\F_q)} \mu(f)^{e_1} \mu(f + a_2)^{e_2} \cdots \mu(f + a_m)^{e_m} \bigg\rvert \leq 2mnq^{-\frac{1}{2}} + 3mn^2 q^{-1}.
  \]
\end{thm*}

Theorem \ref{thm:intro-arith} implies that as $q \to \infty$ with $n$ fixed, a stronger bound of $\bigO(q^{-h - 1})$ holds when we average over all the $q^{(m - 1) (h + 1)}$ choices of $a_2, \dots, a_m$ of degree $\leq h$.

\begin{remark}
  For $m \geq 3$, our bound on the $m$-th moment is weaker than the conjectural optimal bound (and can in fact be deduced from the $m = 2$ bound).\footnote{In contrast, the $m > 2$ case of Theorem \ref{thm:intro} itself --- i.e., the statement about cohomology --- does not follow formally from the $m = 2$ case.} For example, let us give the heuristic for the M\"obius function. We bound the expected value of $\psi_\mu(-, h)^m$, averaged over the $q^{n - h - 1}$ intervals in $M_n$, by
  \[
  \frac{1}{q^{n - h - 1}} \sum_{f_i \sim f_j} \mu(f_1) \cdots \mu(f_m) = \bigO(q^{(m - 1)(h + 1)}) = \bigO(H^{m - 1}).
  \]
  However, $\lvert\psi_\mu(-, h)\rvert$ should conjecturally be of order $H^{1/2}$ on average. This follows from the heuristic of randomly choosing the value of $\mu$ by a fair coin flip for each squarefree polynomial. Hence, we expect the actual order of growth of the expected value of $\lvert \psi_\mu(-, h) \rvert^m$ to be $H^{m/2} = q^{m (h + 1)/2}$, which is smaller than $H^{m - 1}$ for $m \geq 3$.

  This extra cancellation should come not from outright vanishing of the cohomology $H^r(X_{m, n, h}, \Q_\ell)$ in the appropriate degrees, but from the $S_n^m$-action on cohomology, which we expect is nontrivial and interesting in the range not covered by Proposition \ref{prop:sym-action}.
\end{remark}

\section{The geometry of short intervals}
\label{section:geometry}
We now study the geometry of a variety $X_{m, n, h}$ parametrizing (ordered) roots of tuples of polynomials $(f_1, \dots, f_m)$ with $f_i \in M_n$ and $\deg(f_i - f_j) \leq h$ for $i \neq j$. In this section, all varieties are over an algebraically closed field $k$.

\subsection{Definition of the variety}
\label{section:def-variety}
Let $e_n^j(x_1, \dots, x_n)$ be the elementary symmetric polynomial of degree $j$ in the $n$ variables $x_1, \dots, x_n$. Define $X_{m, n, h} \subset \proj^{mn - 1}$ in coordinates $z_{i,j}$ (with $1 \leq i \leq m$, $1 \leq j \leq n$) by the equations
\[
e_n^j(z_{1, 1}, \dots, z_{1, n}) = e_n^j(z_{i, 1}, \dots, z_{i, n})
\]
for all integers $2 \leq i \leq m$ and $1 \leq j \leq n - h - 1$. Thus, a point of $X_{m, n, h}$ is given by
\[
[z_{1, 1} : \dots : z_{1, n} : z_{2, 1} : \dots : z_{2, n} : \dots \dots : z_{m, 1} : \dots : z_{m, n}] \in \proj^{mn - 1}
\]
such that the polynomials in $x$ given by
\begin{equation}
  \label{eqn:eval}
  f_i(x) = (x - z_{i, 1}) (x - z_{i, 2}) \cdots (x - z_{i, n})
\end{equation}
satisfy $\deg(f_i - f_j) \leq h$.

The set $M_n^m$ of $m$-tuples of degree $n$ polynomials can be considered as an affine variety, isomorphic to $\aff^{mn}$. For $1 \leq i \leq m$ and $1 \leq j \leq n$, we consider the $(i, j)$-th coordinate $c_{i, j}$ of $(f_1, \dots, f_m) \in M_n^m$ to be the coefficient of the degree $n - j$ term of $f_i$.

Let $Y_{m, n, h}$ denote the affine cone of $X_{m, n, h}$. There is a map
\[
\pi: Y_{m, n, h} \to M_n^m
\]
sending $(z_{i, j})$ to $(f_1, \dots, f_m) \in M_n^m$ as in \eqref{eqn:eval} above. The image of this map is
\[
\set{(f_1, \dots, f_m) \in M_n^m: \deg(f_i - f_j) \leq h\ \forall i \neq j},
\]
which is isomorphic to an $n + (m - 1) (h + 1)$-plane in $M_n^m \cong \aff^{mn}$. The map $\pi$ is finite with fibers of cardinality $\leq n!$; the preimage of a point $(f_1, \dots, f_m) \in M_n^m$ consists of all permutations of the roots of the $f_i$.

In particular, $\pi$ factors through the quotient $Y_{m, n, h}/S_n^m$, where the $i$-th copy of $S_n$ acts by permuting $(z_{i, 1}, \dots, z_{i, n})$.

\begin{prop}
  $X_{m, n, h}$ is a complete intersection of dimension $n + (m - 1) (h + 1) - 1$.
\end{prop}
\begin{proof}
  Since $X_{m, n, h}$ is defined by $(m - 1) (n - h - 1) = mn - n - (m - 1) (h + 1)$ equations in $\proj^{mn - 1}$, it suffices to show that $X_{m, n, h}$ is of dimension $n + (m - 1) (h + 1) - 1$. Let $Y_{m, n, h} \subset \aff^{mn}$ be the affine cone of $X_{m, n, h}$. By the above discussion, $\pi: Y_{m, n, h} \to M_n^m$ is finite and has $n + (m - 1) (h + 1)$-dimensional image, so we are done.
\end{proof}

\subsection{The singular locus}
In this section, we compute the codimension of the singular locus of $X_{m, n, h}$ to be $2h + 3$. This is the key geometric result needed for our estimates. We will prove the result for $m = 2$, then use this to generalize to arbitrary $m$ when $\chr(k) > n$.

We begin with an elementary lemma on determinants of matrices that appear as blocks in the Jacobian of $X_{m, n, h}$.

\begin{lemma}
  \label{lem:det}
  Let $e_{r, j}^d$ be the elementary symmetric polynomial of degree $d$ in the $r - 1$ variables $x_1, \dots, x_r$, excluding $x_j$. For $s \leq r$, denote the $s \times r$ matrix
  \[
  A_r^s(x_1, \dots, x_r) := \begin{pmatrix}
    e_{r, 1}^0 & e_{r, 2}^0 & \cdots & e_{r, r}^0 \\
    e_{r, 1}^1 & e_{r, 2}^1 & \cdots & e_{r, r}^1 \\
    \vdots & \vdots & \ddots & \vdots \\
    e_{r, 1}^{s - 1} & e_{r, 2}^{s - 1} & \cdots & e_{r, r}^{s - 1}
  \end{pmatrix}.
  \]
  For $1 \leq i_1 < i_2 < \dots < i_s \leq r$, let $B_{i_1, \dots, i_r}$ be the $s \times s$ submatrix of $A_r^s$ comprised of columns $i_1, i_2, \dots, i_s$. Then
  \[
  \det B_{i_1, \dots, i_s}(x_1, \dots, x_r) = \prod_{1 \leq \alpha < \beta \leq s} (x_{i_\alpha} - x_{i_\beta}).
  \]
\end{lemma}
\begin{proof}
  Observe that $\det B_{i_1, \dots, i_s}(x_1, \dots, x_r)$ is a polynomial $D(x_1, \dots, x_r)$ in $x_1, \dots, x_r$ of degree
  \[
  \sum_{d=0}^{s} d = \frac{(s - 1) s}{2} = \binom{s}{2}.
  \]
  If $x_{i_\alpha} = x_{i_\beta}$ for some $1 \leq \alpha < \beta \leq s$, then the $i_\alpha$-th and $i_\beta$-th columns of $B_{i_1, \dots, i_s}$ are equal. Thus, the determinant is divisible by $x_{i_\alpha} - x_{i_\beta}$, and hence by $\prod_{\alpha < \beta} (x_{i_\alpha} - x_{i_\beta})$. But this is also of degree $\binom{s}{2}$, so
  \[
  D(x_1, \dots, x_r) = C \prod_{1 \leq \alpha < \beta \leq s} (x_{i_\alpha} - x_{i_\beta})
  \]
  for some constant $C$. Using the cofactor expansion, the $x_{i_2} x_{i_3}^2 x_{i_4}^3 \cdots x_{i_s}^{s - 1}$ term is
  \[
  \prod_{\alpha=1}^{s - 1} (-1)^{s - \alpha} x_{i_{\alpha + 1}} x_{i_{\alpha + 2}} \cdots x_{i_s} = (-1)^{\frac{(s - 1) s}{2}} x_{i_2} x_{i_3}^2 x_{i_4}^3 \cdots x_{i_s}^{s - 1}.
  \]
  This is the same as the $x_{i_2} x_{i_3}^2 \cdots x_{i_s}^{s - 1}$ term of $\prod_{\alpha < \beta} (x_{i_\alpha} - x_{i_\beta})$, so $C = 1$.
\end{proof}

\begin{remark}
  Note the similarity to the Vandermonde determinant. The above proof is adapted from a computation of the Vandermonde determinant in \cite{determinant-calculus}.
\end{remark}

\begin{cor}
  \label{cor:full-rank}
  For $s \leq r$, the matrix $A_r^s(x_1, \dots, x_r)$ defined in Lemma \ref{lem:det} has full rank if and only if $\#\set{x_1, \dots, x_r} \geq s$.
\end{cor}

Denote $A_i := A_{n}^{n - h - 1}(z_{i, 1}, \dots, z_{i, n})$. The Jacobian of $X_{m, n, h}$ is the block matrix
\[
J_{m, n, h} := \begin{pmatrix}
  A_1 & -A_2 & 0 & 0 & \cdots & 0 \\
  A_1 & 0 & -A_3 & 0 & \cdots & 0 \\
  A_1 & 0 & 0 & -A_4 & \cdots & 0 \\
  \vdots & \vdots & \vdots & \vdots & \ddots & \vdots \\
  A_1 & 0 & 0 & 0 & \cdots & -A_m
\end{pmatrix}.
\]

\begin{prop}
  \label{prop:codim-2}
  The singular locus of $X_{2, n, h}$ has codimension $2h + 3$.
\end{prop}
\begin{proof}
  For the upper bound on dimension, we study the Jacobian matrix of $X_{2, n, h}$ and prove a necessary condition for a point to lie in the singular locus. For the lower bound, we find a subvariety of the appropriate dimension contained in the singular locus.

  First, we show that the singular locus has dimension at most $n - h - 3$. Note that the Jacobian matrix $J_{2, n, h} = \begin{pmatrix}
    A_1 & -A_2
  \end{pmatrix}$ has full rank if and only if $\begin{pmatrix}
    A_1 & A_2
  \end{pmatrix}$ has full rank. For any $1 \leq j \leq n$ and $1 \leq d \leq n - h - 1$,
  \[
  e_{n, j}^d = e_n^d - x_j e_{n, j}^{d - 1} = e_n^d - x_j e_n^{d - 1} + x_j^2 e_{n, j}^{d - 2} = \dots = \sum_{c=0}^k (-x_j)^d e_n^{d - c}.
  \]
  Thus, we can rewrite $A_1$ and $A_2$ as
  \[
  A_i = \begin{pmatrix}
    1 & \cdots & 1 \\
    e_n^1 - z_{i, 1} & \cdots & e_n^1 - z_{i, n} \\
    e_n^2 - z_{i, 1} e_n^1 + z_{i, 1}^2 & \cdots & e_n^2 - z_{i, n} e_n^1 + z_{i, n}^2 \\
    \vdots & \ddots & \vdots \\
    \sum_{d=0}^{n - h - 2} (-z_{i, 1})^d e_n^{n - h - 2 - d} & \cdots & \sum_{d=0}^{n - h - 2} (-z_{i, n})^d e_n^{n - h - 2 - d}
  \end{pmatrix}.
  \]
  If $z_{i, j} = z_{i, j'}$ with $j \neq j'$, then columns $j$ and $j'$ of $A_i$ are equal. Moreover, by the defining equations of $X_{2, n, h}$, the elementary symmetric polynomials $e_n^i$ are the same for $A_1$ and $A_2$; thus, if $z_{1, j} = z_{2, j'}$ for any $j, j'$, then column $j$ of $A_1$ and column $j'$ of $A_2$ are equal.

  Hence, the same argument as for Lemma \ref{lem:det} shows that $J_{2, n, h}$ has less than full rank if and only if
  \[
  \#\set{z_{1, 1}, \dots, z_{1, n}, z_{2, 1}, \dots, z_{2, n}} \leq n - h - 2.
  \]
  The subvariety of $\proj^{2n - 1}$ defined by the above condition has dimension $n - h - 3$, so its intersection with $X_{2, n, h}$ has dimension $\leq n - h - 3$, i.e., codimension at least $2h + 3$.

  To show that $2h + 3$ is \emph{exactly} the codimension of the singular locus, let $S$ be the subvariety of $\proj^{2n - 1}$ defined by equations
  \begin{align*}
    z_{1, j} &= z_{2, j}, & 1 \leq j \leq n; \\
    z_{1, n - h - 2} &= z_{1, j}, & n - h - 1 \leq j \leq n.
  \end{align*}
  Note that $S \subseteq X_{2, n, h}$. Since this variety is defined by $n + h + 2$ equations, its dimension is at least
  \[
  2n - 1 - (n + h + 2) = n - h - 3,
  \]
  i.e., codimension at most $2h + 3$ in $X_{2, n, h}$. Moreover, at any point of $S$, we have $A_1 = A_2$ and $\#\set{z_{1, 1}, \dots, z_{1, n}} \leq n - h - 2$, so by Corollary \ref{cor:full-rank}, the first $n - h - 1$ rows of $J$ are linearly dependent. Thus, $S$ is contained in the singular locus of $X_{2, n, h}$.
\end{proof}

To generalize to arbitrary $m \geq 2$, we will need the following lemma, which intuitively says that conditions on the higher-degree coefficients of a polynomial are algebraically independent of conditions on multiplicity of the roots.

\begin{lemma}
  \label{lem:sing-pts}
  Suppose $\chr(k) > n$, and fix $1 \leq s, t \leq n$ and $c_1, c_2, \dots, c_t \in k$. Let $Z \subseteq \aff^n$ be the set of all $(w_1, \dots, w_n) \in \aff^n$ such that $\#\set{w_1, \dots, w_n} \leq s$ and $e_n^j(w_1, \dots, w_n) = c_j$ for each $1 \leq j \leq t$. If $Z$ is non-empty, then $\dim(Z) = \max \set{s - t, 0}$.
\end{lemma}
\begin{proof}
  Our strategy is to stratify $Z$ into finitely many pieces based on the pattern of repetitions among the $w_i$. We then show that each piece has Jacobian of full rank at each point, which implies each piece is nonsingular of dimension at most $\max \set{s - t, 0}$.

  Since $\chr(k) > n$, by Newton's identities, the system of equations
  \[
  e_n^j(w_1, \dots, w_n) = c_j
  \]
  is equivalent to
  \[
  w_1^j + w_2^j + \dots + w_n^j = d_j
  \]
  for each $1 \leq j \leq s$ and some constants $d_j$, where $d_j$ is a polynomial in $c_1, \dots, c_j$ with zero constant term.

  For integers $1 \leq r \leq s$ and $1 \leq \alpha_1, \dots, \alpha_r \leq n$ such that $\alpha_1 + \dots + \alpha_r = n$, let $W := W_{r; \alpha_1, \dots, \alpha_r}$ be the subvariety
  \[
  \set{(u_1, \dots, u_r) \in \aff^r \mathrel{\bigg{|}} \begin{aligned}
    \alpha_1 u_1^j + \dots + \alpha_r u_r^j = d_j && \forall 1 \leq j \leq t, \\
    u_i \neq u_{i'} && \forall i \neq i'
  \end{aligned}}.
  \]
  The tangent space to any point $(v_1, \dots, v_r) \in W$ is the kernel of the $r \times t$ Jacobian matrix $J_W(v_1, \dots, v_r)$, whose $(i, j)$-th entry is $\alpha_i j v_i^{j - 1}$. Note that $\alpha_i j \neq 0$ since $\chr(k) > n$, so $J_W(v_1, \dots, v_r) = PVQ$, where $P$ and $Q$ are diagonal invertible square matrices and the maximal upper-left submatrix of $V$ is a Vandermonde matrix, whose determinant vanishes if and only if $v_i = v_{i'}$ for some $i \neq i'$.

  By construction of $W$, we have $v_i \neq v_{i'}$ for all $i \neq i'$, so the Jacobian of $W$ has full rank $\min \set{r, t}$; hence, $\dim T_{(v_1, \dots, v_r)} W = r - \min \set{r, t} = \max \set{r - t, 0}$. This is true for every $(v_1, \dots, v_r) \in W$, so $W$ is nonsingular of pure dimension $\max \set{r - t, 0}$.

  For any $1 \leq i_1 < i_2 < \dots < i_r \leq n$ and $\alpha_1, \dots, \alpha_r$ as above, let
  \[
  Z_{r; \alpha_1, \dots, \alpha_r}^{i_1, \dots, i_r} := \set{(w_1, \dots, w_n) \in Z \mathrel{\bigg{|}} \begin{aligned}
    &\#\set{w_{i_1}, \dots, w_{i_r}} = \#\set{w_1, \dots, w_n} = r, \\
    &\#\set{j: w_j = w_{i_\nu}} = \alpha_\nu \quad \forall 1 \leq \nu \leq r
  \end{aligned}}.
  \]
  The map
  \begin{align*}
    Z_{r; \alpha_1, \dots, \alpha_r}^{i_1, \dots, i_r} &\to W_{r; \alpha_1, \dots, \alpha_r} \\
    (w_1, \dots, w_n) &\mapsto (w_{i_1}, \dots, w_{i_r})
  \end{align*}
  has finite fibers (because each $w_j$ with $j \notin \set{i_1, \dots, i_r}$ is equal to one of the $w_{i_\nu}$), so
  \[
  \dim Z_{r; \alpha_1, \dots, \alpha_r}^{i_1, \dots, i_r} = \dim W_{r; \alpha_1, \dots, \alpha_r} = \max \set{r - t, 0}.
  \]
  We can stratify $Z$ into a finite union
  \[
  Z = \bigcup_{\substack{1 \leq r \leq s \\ \alpha_1 + \dots + \alpha_r = n \\ 1 \leq i_1 < \dots < i_r \leq n}} Z_{r; \alpha_1, \dots, \alpha_r}^{i_1, \dots, i_r},
  \]
  so $\dim(Z) = \max(\dim Z_{r; \alpha_1, \dots, \alpha_r}^{i_1, \dots, i_r}) = \max \set{s - t, 0}$.
\end{proof}

\begin{thm}
  \label{thm:codim}
  Suppose $\chr(k) > n$. For $m \geq 2$, the singular locus of $X_{m, n, h}$ has codimension $2h + 3$.
\end{thm}
\begin{proof}
  Let $S$ be the subvariety of $\proj^{mn - 1}$ defined by equations
  \begin{align*}
    e_n^j(z_{1, 1}, \dots, z_{1, n}) &= e_n^j(z_{i, 1}, \dots, z_{i, n}), & 3 \leq i \leq m,\ 1 \leq j \leq n - h - 1; \\
    z_{1, j} &= z_{2, j}, & 1 \leq j \leq n; \\
    z_{1, n - h - 2} &= z_{1, j}, & n - h - 1 \leq j \leq n.
  \end{align*}
  Note that $S \subseteq X_{m, n, h}$. Since this variety is defined by $(m - 2) (n - h - 1) + n + (h + 2)$ equations, it has dimension at least
  \[
  mn - 1 - (m - 2) (n - h - 1) - n - (h + 2) = n + (m - 3) (h + 1) - 2.
  \]
  Moreover, at any point of $S$, $A_1 = A_2$ and $\#\set{z_{1, 1}, \dots, z_{1, n}} \leq n - h - 2$, so by Corollary \ref{cor:full-rank}, the first $n - h - 1$ rows of $J$ are linearly dependent. Thus, $S$ is contained in the singular locus of $X_{m, n, h}$, giving the desired lower bound on dimension.

  Now we show the upper bound. We do this by showing that any point in the singular locus must have a certain number of repeated roots: more precisely, $\# \set{z_{i, 1}, \dots, z_{i, n}} \leq n - h - 2$ for at least two $i$. The desired bound is then a consequence of Lemma \ref{lem:sing-pts}. Let
  \[
  T_{n - h - 2} := \set{(w_1, \dots, w_n) \in \aff^n: \#\set{w_1, \dots, w_n} \leq n - h - 2}.
  \]
  Let $Y_{m, n, h}$ be the affine cone of $X_{m, n, h}$, and for each $\iota = 1, \dots, m$, consider the projection
  \begin{align*}
    p_\iota: Y_{m, n, h} &\to \aff^n, \\
    (z_{i, j}) &\mapsto (z_{\iota, 1}, \dots, z_{\iota, n}).
  \end{align*}
  This map is surjective, and each fiber has dimension $(m - 1) (h + 1)$. In particular, for fixed $\mathbf{w} = (w_1, \dots, w_n) \in \aff^n$, let
  \[
  Z_{\mathbf{w}} := \set{\mathbf{x} = (x_1, \dots, x_n) \in \aff^n: e_n^j(\mathbf{x}) = e_n^j(\mathbf{w})\ \forall j \leq n - h - 1}.
  \]
  Then $\dim Z_{\mathbf{w}} = h + 1$, and
  \[
  p_\iota^{-1}(w_1, \dots, w_n) \cong Z_{\mathbf{w}}^{m - 1}.
  \]

  At any point $(z_{i, j}) \in \Sing(Y_{m, n, h})$, there is a linear dependence relation between the rows of $J := J_{m, n, h}$, involving the rows of $A_1$ and one or more $A_\iota$ with $\iota \geq 2$. For $2 \leq \iota \leq m$, let $S_\iota \subseteq \Sing(Y_{m, n, h})$ be the locus on which there is such a relation involving the rows of $A_\iota$. By Corollary \ref{cor:full-rank}, for all $(z_{i, j}) \in S_\iota$, $(z_{\iota, 1}, \dots, z_{\iota, n}) \in T_{n - h - 2}$. Moreover, if $(z_{i, j}) \in S_\iota \setminus \bigcup_{\iota' \neq \iota} S_{\iota'}$, then Corollary \ref{cor:full-rank} also implies that $(z_{1, 1}, \dots, z_{1, n}) \in T_{n - h - 2}$.

  For any $\mathbf{w} \in \aff^n$ and $1 \leq \ell \leq m$ with $\ell \neq \iota$, let
  \[
  S_{\iota, \ell, \mathbf{w}} := p_\iota^{-1}(w_1, \dots, w_n) \cap \set{(z_{i, j}) \in Y_{m, n, h}: (z_{\ell, 1}, \dots, z_{\ell, n}) \in T_{n - h - 2} \cap Z_{\mathbf{w}}}.
  \]
  By the previous paragraph, for each $(z_{i, j}) \in S_\iota$, we have $(z_{\iota, 1}, \dots, z_{\iota, n}) \in T_{n - h - 2}$ and $(z_{\ell, 1}, \dots, z_{\ell, n}) \in T_{n - h - 2}$ for some $\ell \neq \iota$. Thus,
  \begin{equation}
    \label{eqn:sing-decomp}
    S_\iota \subseteq \bigcup_{\mathbf{w} \in T_{n - h - 2}} \bigcup_{\substack{1 \leq \ell \leq m \\ \ell \neq \iota}} S_{\iota, \ell, \mathbf{w}}.
  \end{equation}
  Moreover, $S_{\iota, \ell, \mathbf{w}}$ is isomorphic to $(T_{n - h - 2} \cap Z_{\mathbf{w}}) \times Z_{\mathbf{w}}^{m - 2}$. By Lemma \ref{lem:sing-pts} with $s = n - h - 2$ and $t = n - h - 1$, for any $\mathbf{w} \in \aff^n$, $T_{n - h - 2} \cap Z_{\mathbf{w}}$ is finite. Hence,
  \[
  \dim S_{\iota, \ell, \mathbf{w}} = (m - 2) (h + 1).
  \]
  Combining this with \eqref{eqn:sing-decomp} and counting dimension, we obtain
  \[
  \dim S_\iota \leq \dim T_{n - h - 2} + \dim S_{\iota, \ell, \mathbf{w}} = (n - h - 2) + (m - 2) (h + 1).
  \]
  Thus, $\codim S_\iota \geq 2h + 3$. Since $\Sing(Y_{m, n, h}) = \bigcup_{\iota=2}^{m} S_\iota$, this completes the proof.
\end{proof}

\subsection{Irreducibility}
Finally, we prove $X_{m, n, h}$ is irreducible, which is essential for the cohomological results of the next section. As before, for each $\mathbf{w} = (w_1, \dots, w_n) \in \aff^n$, define
\[
Z_{\mathbf{w}} := \set{\mathbf{z} = (z_1, \dots, z_n) \in \aff^n: e_n^j(\mathbf{z}) = e_n^j(\mathbf{w}),\ j = 1, \dots, n - h - 1}.
\]

\begin{lemma}
  \label{lem:irred}
  For each $\mathbf{w} = (w_1, \dots, w_n) \in \aff^n$, the variety $Z_{\mathbf{w}}$ is irreducible.
\end{lemma}
\begin{proof}
  Let $\mathcal{F} = x^n + \sum_{j=1}^{n - h - 1} e_n^j(\mathbf{w}) x^{n - j} + \sum_{j=0}^{h} A_j x^j \in k[A_0, \dots, A_h, x]$. By \cite[Lemma 3.2]{BBSR}, $\mathcal{F}$ is separable in $x$ and irreducible in $k(A_0, \dots, A_h)[x]$. Thus,
  \[
  V = \set{(A_0, \dots, A_h, z_1, \dots, z_n) \in \aff^{h + 1} \times \aff^n: \prod_{i=1}^{n} (x - z_i) = \mathcal{F}(A_0, \dots, A_h, x)}
  \]
  is an irreducible subvariety of $\aff^{h + 1} \times \aff^n$. If $(A_0, \dots, A_h, z_1, \dots, z_n) \in V$, then for each $j \leq n - h - 1$, $e_n^j(z_1, \dots, z_n)$ is the degree $n - j$ coefficient of $\prod_{i=1}^{n} (x - z_i) = \mathcal{F}(A_0, \dots, A_h, x)$, namely $e_n^j(\mathbf{w})$. Hence the projection $\aff^{h + 1} \times \aff^n \to \aff^n$ restricts to a map $V \to Z_{\mathbf{w}}$. Conversely, the map
  \begin{align*}
    \aff^n &\to \aff^{h + 1} \times \aff^n \\
    \mathbf{z} = (z_1, \dots, z_n) &\mapsto (e_n^n(\mathbf{z}), e_n^{n - 1}(\mathbf{z}), \dots, e_n^{n - h}(\mathbf{z}), z_1, \dots, z_n)
  \end{align*}
  restricts to a map $Z_{\mathbf{w}} \to V$, and these two maps are inverse to each other, proving irreducibility of $Z_{\mathbf{w}}$.
\end{proof}

\begin{lemma}
  \label{lem:flatness}
  Let $f: Y \to S$ be a morphism of locally Noetherian schemes. Assume $S$ is regular, $Y$ is Cohen--Macaulay, and
  \[
  \dim_y Y = \dim_{f(y)} S + \dim_y f^{-1}(f(y))
  \]
  for all $y \in Y$. Then $f$ is flat.
\end{lemma}
\begin{proof}
  Immediate from \cite[Prop.\ 6.1.5]{EGA-IV.2}.
\end{proof}

\begin{prop}
  $X_{m, n, h}$ is irreducible.
\end{prop}
\begin{proof}
  We again pass to the affine cone $Y := Y_{m, n, h}$. Consider the projection map
  \[
  \pi: (z_{i, j}) \mapsto (z_{1, 1}, \dots, z_{1, n}): Y \to \aff^n.
  \]
  This is surjective, and the fiber above a point $\mathbf{w} = (w_1, \dots, w_n) \in \aff^n$ is the direct product of $m - 1$ copies of $Z_{\mathbf{w}}$, which is irreducible by Lemma \ref{lem:irred}. Hence, by \cite[tag 004Z]{stacks}, if $\pi$ is open, then $Y$ is irreducible. By \cite[tag 01UA]{stacks}, it suffices to show that $\pi$ is flat. Since $Y$ is a complete intersection, $Y$ is Cohen--Macaulay, so flatness of $\pi$ follows from Lemma \ref{lem:flatness}.
\end{proof}

\section{Actions on cohomology}
\label{section:cohomology}
\subsection{Cohomology of complete intersections}
The results of the previous section on the dimension of $X_{m, n, h}$ and its singular locus determine the $\ell$-adic cohomology groups $H^r(X_{m, n, h}, \Q_\ell)$ in higher degrees. This is due to the following theorem of Ghorpade and Lachaud.

\begin{thm}[\cite{sing-var}, Prop.\ 3.3]
  \label{thm:comp-int-cohom}
  Let $k$ be an algebraically closed field, let $\ell \neq \chr(k)$ be prime, and let $X$ be a complete intersection in $\proj_k^N$ of dimension $d \geq 1$ with $\dim \Sing(X) \leq s$. Then, for $0 \leq r \leq d - 1$ or $d + s + 2 \leq r \leq 2d$,
  \[
  H^r(X, \Q_\ell) = \begin{cases}
    \Q_\ell(-r/2) & \text{if $r$ is even}, \\
    0 & \text{if $r$ is odd}.
  \end{cases}
  \]
  Furthermore, if $d + s + 1$ is even, then $H^{d + s + 1}(X, \Q_\ell)$ contains a subspace isomorphic to $\Q_\ell(-(d + s + 1)/2)$.
\end{thm}

\begin{remark}
  This has a topological analogue, a ``partial Poincar\'e duality'' for singular spaces due to Kato \cite{Kato}.
\end{remark}

Specializing to $X = X_{m, n, h}$:
\begin{cor}
  \label{cor:cohom}
  Suppose $\chr(k) > n$ or $m = 2$. For $2n + 2 (m - 2) (h + 1) - 1 \leq r \leq 2n + 2 (m - 1) (h + 1) - 2$,
  \[
  H^r(X_{m, n, h}, \Q_\ell) = \begin{cases}
    \Q_\ell(-r/2) & \text{if $r$ is even}, \\
    0 & \text{if $r$ is odd}.
  \end{cases}
  \]
  Also, $H^{2n + 2 (m - 2) (h + 1) - 2}(X_{m, n, h}, \Q_\ell)$ contains a subspace isomorphic to
  \[
  \Q_{\ell}(-n - (m - 2) (h + 1) + 1).
  \]
\end{cor}
\begin{proof}
  Upon noting that
  \begin{align*}
    \dim X + \dim \Sing(X) + 2 &= 2 \dim(X) - (2h + 3) + 2 \\
    &= 2 (n + (m - 1) (h + 1) - 1) - 2 (h + 1) - 1 + 2 \\
    &= 2n + 2 (m - 2) (h + 1) - 1,
  \end{align*}
  this follows immediately from Theorems \ref{thm:codim} and \ref{thm:comp-int-cohom}.
\end{proof}

\subsection{The permutation action}
We will count points where the action of Frobenius induces various permutations of the coordinates. To do so, we show that $S_n^m$ acts trivially on cohomology in the degrees we have computed.

\begin{lemma}
  \label{lem:proj-sym-action}
  For all $N \geq 1$ and all $r \in \Z$, let $S_N$ act on $\proj^{N - 1}$ by permuting the coordinates. The action of $S_N$ on $H^r(\proj^{N - 1}, \Q_\ell)$ is trivial.
\end{lemma}
\begin{proof}
  For any $\sigma \in S_N$ and $r$ even with $0 \leq r \leq 2(N - 1)$, let $P$ be an $(r/2)$-dimensional hyperplane in $\proj^{N - 1}$ preserved by $\sigma$. This induces an isomorphism $H^r(\proj^{N - 1}, \Q_\ell) \to H^r(P, \Q_\ell)$, and $H^r(P, \Q_\ell)$ is Poincar\'e dual to $H^0(P, \Q_\ell)$, on which $\sigma$ acts trivially.
\end{proof}

\begin{prop}
  \label{prop:sym-action}
  Suppose $\chr(k) > n$ or $m = 2$. Let $S_n^m$ act on $X_{m, n, h}$ by the $i$-th copy of $S_n$ permuting $(z_{i, 1}, \dots, z_{i, n})$. For $r \geq 2n + 2 (m - 2)(h + 1) - 1$, the action of $S_n^m$ on $H^r(X_{m, n, h}, \Q_\ell)$ is trivial.
\end{prop}
\begin{proof}
  The degree $n - h - 1$ Veronese embedding realizes the inclusion $X_{m, n, h} \inject X_{m, n, h + 1}$ as a proper linear section of codimension $m - 1$ which is preserved by the $S_n^m$-action. By \cite[Prop.\ 2.5]{sing-var}, since $X_{m, n, h}$ is regular in codimension one, $(X_{m, n, h + 1}, X_{m, n, h})$ is a ``semi-regular pair'' in the terminology of loc.\ cit. Hence, \cite[Thm.\ 2.4]{sing-var} combined with Theorem \ref{thm:codim} gives a canonical, Galois-equivariant isomorphism
  \[
  H^r(X_{m, n, h}, \Q_\ell) \cong H^{r + 2(m - 1)}(X_{m, n, h+1}, \Q_\ell(m - 1))
  \]
  for all $r \geq 2 \dim X_{m, n, h} - (2h + 3) + 2 = 2n + 2 (m - 2)(h + 1) - 1$.

  Since $r + 2(m - 1) \geq 2 \dim X_{m, n, h+1} - (2(h + 1) + 3) + 2$, we may apply the same argument with $h$ replaced by $h + 1$. By induction, we are reduced to the case $h = n - 3$, where the same argument applies to give an isomorphism
  \[
  H^r(X_{m, n, n - 3}, \Q_\ell) \cong H^{r + 2(m - 1)}(\proj^{mn - m - 2}, \Q_\ell(m - 1)),
  \]
  noting that $X_{m, n, n - 3} \subseteq \proj^{mn - 1}$ is defined by $m - 1$ linear forms and $m - 1$ quadratic forms. By Lemma \ref{lem:proj-sym-action}, $S_n^m$ acts trivially on the cohomology of projective space, so we are done.
\end{proof}

\section{Arithmetic functions}
\label{section:arithmetic}
The above results in cohomology imply asymptotic bounds on the sum
\[
\frac{1}{q^n} \sum_{f \in M_n(\F_q)} \Bigl( \sum_{\substack{g \in \F_q[x] \\ \deg g \leq h}} \varphi(f + g) \Bigr)^m = q^{h + 1} \cdot \frac{1}{q^n} \sum_{\substack{f_1, \dots, f_m \in M_n(\F_q) \\ f_i \sim f_j}} \varphi(f_1) \cdots \varphi(f_m)
\]
for many arithmetic functions $\varphi$. We prove the general result, then apply it to the von Mangoldt function $\Lambda$ and the M\"obius function $\mu$.

\begin{remark}
  \label{rem:char-restriction}
  For the remainder of the paper, our point of view is that $m$, $n$, and $h$ are fixed and $q$ is growing, ranging over powers of primes; all asymptotic notation has implicit constants that may depend on $m$, $n$, and $h$.

  If $m \geq 3$, we additionally restrict $q$ to powers of primes $p > n$. We suspect that this restriction (due solely to the failure of Lemma \ref{lem:sing-pts} in low characteristic) could be avoided by more refined reasoning about the singular locus of $X_{m, n, h}$.
\end{remark}

\begin{lemma}
  \label{lem:pt-count}
  For any $\sigma = (\sigma_1, \dots, \sigma_m) \in S_n^m$,
  \begin{multline*}
    \#\set{(z_{i, j}) \in Y_{m, n, h}: \Frob_q(z_{i, j}) = \sigma_i(z_{i, j})\ \forall i, j} \\
    = q^{n + (m - 1) (h + 1)} + \bigO(q^{n + (m - 2) (h + 1)}).
  \end{multline*}
\end{lemma}
\begin{proof}
  By the Grothendieck--Lefschetz trace formula,
  \begin{multline*}
    \#\set{[z_{i, j}] \in X_{m, n, h}: \Frob_q(z_{i, j}) = \sigma_i(z_{i, j})\ \forall i, j} \\
    = \sum_{r} (-1)^r \Tr\bigl( \Frob_q \circ \sigma^{-1}: H^r(X_{m, n, h}, \Q_\ell) \bigr).
  \end{multline*}
  For all $r \geq 2n + 2 (m - 2)(h + 1) - 1$, by Proposition \ref{prop:sym-action}, $\sigma$ acts trivially on $H^r(X_{m, n, h}, \Q_\ell)$, so
  \[
  \Tr\bigl( \Frob_q \circ \sigma^{-1}: H^r(X_{m, n, h}, \Q_\ell) \bigr) = \Tr\bigl( \Frob_q: H^r(X_{m, n, h}, \Q_\ell) \bigr).
  \]
  By Corollary \ref{cor:cohom},
  \begin{align*}
    &\sum_{r} (-1)^r \Tr\bigl( \Frob_q: H^r(X_{m, n, h}, \Q_\ell) \bigr) \\
    &= \sum_{j=1}^{h + 1} q^{n + (m - 1) (h + 1) - j} + \bigO(q^{n + (m - 2) (h + 1) - 1}) \\
    &= \frac{1}{q - 1} \left( q^{n + (m - 1) (h + 1)} + \bigO(q^{n + (m - 2) (h + 1)}) \right).
  \end{align*}
  Passing to the affine cone $Y_{m, n, h}$, we obtain the result.
\end{proof}

Now we can now prove our main result.

\begin{defn}
  We say a function $\varphi: M_n(\F_q) \to \R$ is \emph{arithmetic of von Mangoldt type} with coefficients $(c_\sigma)_{\sigma \in S_n}$ if $\varphi$ is given by
  \[
  \varphi(f) = \sum_{\substack{(z_1, \dots, z_n) \in \aff^n \\ f = \prod_i (x - z_i)}} \sum_{\sigma \in S_n} c_\sigma \delta_\sigma(z_1, \dots, z_n)
  \]
  for all $f \in M_n(\F_q)$, where
  \[
  \delta_\sigma(z_1, \dots, z_n) = \begin{cases}
    1 & \text{if } \Frob_q(z_1, \dots, z_n) = \sigma(z_1, \dots, z_n), \\
    0 & \text{otherwise}.
  \end{cases}
  \]
\end{defn}

\begin{thm}
  \label{thm:arith-bound}
  Fix $m \geq 2$, $n \geq 4$, and $1 \leq h \leq n - 3$. Let $\varphi_1, \dots, \varphi_m: M_n(\F_q) \to \R$ be arithmetic functions of von Mangoldt type where $\varphi_i$ has coefficients $(c_{i, \sigma})_{\sigma \in S_n}$. Then, in the limit as $q \to \infty$ (with the restrictions of Remark \ref{rem:char-restriction}),
  \[
  \frac{1}{q^n} \sum_{f \in M_n(\F_q)} \prod_{i=1}^{m} \sum_{\substack{g \in \F_q[x] \\ \deg g \leq h}} \varphi_i(f + g) = \prod_{i=1}^{m} \Bigl( \sum_{\tau \in S_n} c_{i, \tau} q^{h + 1} \Bigr) + \bigO(q^{(h + 1) (m - 1)}).
  \]
\end{thm}
\begin{proof}
  By Lemma \ref{lem:pt-count}, for any $\sigma_1, \dots, \sigma_m \in S_n$,
  \[
  \sum_{(z_{i, j}) \in Y_{m, n, h}} \prod_{i=1}^{m} \delta_{\sigma_i}(z_{i, 1}, \dots, z_{i, n}) = q^{n + (m - 1) (h + 1)} + \bigO(q^{n + (m - 2) (h + 1)}).
  \]
  Thus,
  \begin{align*}
    & \frac{1}{q^n} \sum_{f \in M_n(\F_q)} \prod_{i=1}^{m} \sum_{\substack{g \in \F_q[x] \\ \deg g \leq h}} \varphi_i(f + g) \\
    &= \frac{q^{h + 1}}{q^n} \sum_{\substack{f_1, \dots, f_m \in M_n(\F_q) \\ f_i \sim f_j}} \varphi_1(f_1) \cdots \varphi_m(f_m) \\
    &= \frac{q^{h + 1}}{q^n} \sum_{\substack{f_1, \dots, f_m \in M_n(\F_q) \\ f_i \sim f_j}} \sum_{\substack{(z_{i, j}) \in \aff^{mn} \\ f_i = \prod_j (x - z_{i, j})}} \sum_{\sigma = (\sigma_1, \dots, \sigma_m) \in S_n^m} \prod_{i=1}^{m} c_{\sigma_i} \delta_{\sigma_i}(z_{i, 1}, \dots, z_{i, n}) \\
    &= \frac{q^{h + 1}}{q^n} \sum_{\sigma \in S_n^m} \sum_{(z_{i, j}) \in Y_{m, n, h}} \prod_{i=1}^{m} c_{i, \sigma_i} \delta_{\sigma_i}(z_{i, 1}, \dots, z_{i, n}) \\
    &= \frac{q^{h + 1}}{q^n} \sum_{\sigma \in S_n^m} c_{1, \sigma_1} \cdots c_{m, \sigma_m} \sum_{(z_{i, j}) \in Y_{m, n, h}} \prod_{i=1}^{m} \delta_{\sigma_i}(z_{i, 1}, \dots, z_{i, n}) \\
    &= \frac{q^{h + 1}}{q^n} \sum_{\sigma \in S_n^m} c_{1, \sigma_1} \cdots c_{m, \sigma_m} \left( q^{n + (m - 1) (h + 1)} + \bigO_\sigma(q^{n + (m - 2) (h + 1)}) \right) \\
    &= \frac{q^{h + 1}}{q^n} \sum_{\sigma \in S_n^m} c_{1, \sigma_1} \cdots c_{m, \sigma_m} q^{n + (m - 1) (h + 1)} + \bigO(q^{(h + 1) (m - 1)}) \\
    % &= \frac{q^{h + 1}}{q^n} \Bigl( \sum_{\tau \in S_n} c_\tau \Bigr)^m \left( q^{n + (m - 1) (h + 1)} + \bigO(q^{n + (m - 2) (h + 1)}) \right) \\
    &= \Bigl( \prod_{i=1}^{m} \sum_{\tau \in S_n} c_{i, \tau} \Bigr) q^{(h + 1) m} + \bigO(q^{(h + 1) (m - 1)}) \\
    &= \prod_{i=1}^{m} \Bigl( \sum_{\tau \in S_n} c_{i, \tau} q^{h + 1} \Bigr) + \bigO(q^{(h + 1) (m - 1)}).
    \qedhere
  \end{align*}
\end{proof}

Let us also note some basic facts about the set of arithmetic functions of von Mangoldt type.

\begin{lemma}
  \label{lem:Mangoldt-type}
  The set of arithmetic functions $\varphi: M_n(\F_q) \to \R$ of von Mangoldt type is an $\R$-vector space containing the constant functions.
\end{lemma}
\begin{proof}
  Let $f \in M_n(\F_q)$ be arbitrary, and let $z_1, \dots, z_n \in \F_{q^n}$ such that $f = \prod_i (x - z_i)$. Let $G \subseteq S_n$ be the stabilizer of $(z_1, \dots, z_n)$. (E.g., $G = 1$ if and only if $f$ is squarefree.) In the notation of Theorem \ref{thm:arith-bound}, for $\tau \in S_n$, we have $\delta_\tau(z_1, \dots, z_n) = 1$ if and only if $\tau G = \sigma_f G$, where $\sigma_f \in S_n/G$ is the coset given by the action of $\Frob_q$ on $(z_1, \dots, z_n)$. Moreover, there are exactly $\#(S_n/G)$ possible choices of $(z_1, \dots, z_n)$. Hence,
  \[
  \sum_{\substack{(z_1, \dots, z_n) \in \F_{q^n}^n \\ f = \prod_i (x - z_i)}} \sum_{\tau \in S_n} \frac{1}{n!} \delta_\tau(z_1, \dots, z_n) = \frac{1}{n!} \#(S_n/G) \#G = 1.
  \]
  Thus, $1$ is of von Mangoldt type with coefficients $c_\sigma = \frac{1}{n!}$ for all $\sigma \in S_n$.

  Finally, it is immediate from the definition that adding and scaling arithmetic functions of von Mangoldt type produces more such functions with coefficients added and scaled componentwise.
\end{proof}

We can compare this notion to Rodgers' \emph{factorization functions} \cite[2.2]{Rodgers}.

\begin{prop}
  \label{prop:Mangoldt-factorization}
  Let $a: M_n(\F_q) \to \R$ be a \emph{factorization function}, i.e., if $f = P_1^{e_1} \cdot \dots \cdot P_k^{e_k} \in M_n(\F_q)$ with the $P_i$ distinct monic irreducible polynomials, then $a(f)$ depends only on the data $(\deg P_1, e_1; \dots; \deg P_k, e_k)$. There is an arithmetic function of von Mangoldt type $\varphi: M_n(\F_q) \to \R$ such that $\varphi(f) = a(f)$ for all squarefree $f \in M_n(\F_q)$.
\end{prop}
\begin{proof}
  By linearity, it suffices to show the following: For each $\sigma \in S_n$, there is some nonzero $d_\sigma \in \R$ and an arithmetic function of von Mangoldt type $\varphi_\sigma: M_n(\F_q) \to \R$ such that for all \emph{squarefree} $f \in M_n(\F_q)$,
  \[
  \varphi_\sigma(f) = \begin{cases}
    d_\sigma & \text{if $f$ has cycle type $\sigma$}, \\
    0 & \text{otherwise}.
  \end{cases}
  \]
  Choose $\varphi_\sigma$ to have coefficients $c_\sigma = 1$ and $c_\tau = 0$ for each permutation $\tau \neq \sigma$. Then for all $f \in M_n(\F_q)$,
  \[
  \varphi_\sigma(f) = \sum_{\substack{(z_1, \dots, z_n) \in \aff^n \\ f = \prod_i (x - z_i)}} \delta_\sigma(z_1, \dots, z_n).
  \]
  If $f$ is squarefree and does not have cycle type $\sigma$, then $\varphi_\sigma(f) = 0$. If $f$ is squarefree and has cycle type $\sigma$, then $\varphi_\sigma(f) = d_\sigma$ is the number of ways to write $\sigma$ in cycle notation with nonincreasing cycle length (and with $1$-cycles written explicitly). Equivalently, $d_\sigma$ is the order of $S_n$ divided by the order of the conjugacy class of $\sigma$. So this is the function we need.
\end{proof}

\begin{remark}
  Rodgers' factorization functions are slightly more general, because they can be defined independently on the non-squarefree locus. As $\sigma$ ranges over a set of representatives of the conjugacy classes in $S_n$, the functions $\varphi_\sigma$ defined above form a basis of the space of arithmetic functions of von Mangoldt type. In particular, arithmetic functions of von Mangoldt type are determined by their restriction to the squarefree locus.

  As for the non-squarefree locus, we have the bound $\varphi_\sigma(g) \leq \varphi_\sigma(f)$ for any $f, g \in M_n(\F_q)$ with $f$ squarefree of cycle type $\sigma$. (Indeed, since $\varphi_\sigma(g) \geq 0$ for all $g$, this follows from Lemma \ref{lem:Mangoldt-type}.)
\end{remark}

As a side note, since we know the cohomology of $\aff^n$ completely, we can prove an exact result for $m = 1$:
\begin{prop}
  \label{prop:m=1}
  For all $n \geq 1$ and any arithmetic function $\varphi: M_n(\F_q) \to \R$ of von Mangoldt type with coefficients $(c_\sigma)_{\sigma \in S_n}$,
  \[
  \sum_{f \in M_n(\F_q)} \varphi(f) = \sum_{\sigma \in S_n} c_\sigma q^n.
  \]
\end{prop}
\begin{proof}
  By Poincar\'e duality, the compactly supported cohomology of $\aff^n$ is given by
  \[
  H_c^r(\aff^n, \Q_\ell) = \begin{cases}
    \Q_\ell(-n), & r = 2n, \\
    0, & r \neq 2n.
  \end{cases}
  \]
  By the Grothendieck--Lefschetz trace formula, for any $\sigma \in S_n$,
  \begin{align*}
    &\sum_{(z_1, \dots, z_n) \in \aff^n} \delta_\sigma(z_1, \dots, z_n) \\
    &= \#\set{(z_1, \dots, z_n) \in \aff^n: \Frob_q(z_1, \dots, z_n) = \sigma(z_1, \dots, z_n)} \\
    &= \sum_{r} (-1)^r \Tr\bigl( \Frob_q \circ \sigma^{-1}: H_c^r(\aff^n, \Q_\ell) \bigr) = q^n.
  \end{align*}
  Thus,
  \begin{align*}
    \sum_{f \in M_n(\F_q)} \varphi(f) &= \sum_{f \in M_n(\F_q)} \sum_{\substack{(z_1, \dots, z_n) \in \aff^n \\ f = \prod_i (x - z_i)}} \sum_{\sigma \in S_n} c_\sigma \delta_\sigma(z_1, \dots, z_n) \\
    &= \sum_{(z_1, \dots, z_n) \in \aff^n} \sum_{\sigma \in S_n} c_\sigma \delta_\sigma(z_1, \dots, z_n) \\
    &= \sum_{\sigma \in S_n} c_\sigma \sum_{(z_1, \dots, z_n) \in \aff^n} \delta_\sigma(z_1, \dots, z_n) = \sum_{\sigma \in S_n} c_\sigma q^n.
    \qedhere
  \end{align*}
\end{proof}

\subsection{Von Mangoldt sums}
Recall the polynomial von Mangoldt function $\Lambda$, defined by $\Lambda(f) = \deg(g)$ for $f = g^k$ with $g$ irreducible, and $\Lambda(f) = 0$ otherwise.

\begin{cor}
  \label{cor:von-mangoldt}
  For any integers $m \geq 2$, $n \geq 4$, and $1 \leq h \leq n - 3$, in the limit as $q \to \infty$ (with the restrictions of Remark \ref{rem:char-restriction}),
  \[
  \frac{1}{q^n} \sum_{f \in M_n(\F_q)} \Bigl( \sum_{\substack{g \in \F_q[x] \\ \deg g \leq h}} \Lambda(f + g) \Bigr)^m = q^{(h + 1) m} + \bigO(q^{(h + 1) (m - 1)}),
  \]
  with the implicit constant depending on $m$, $n$, and $h$.
\end{cor}
\begin{proof}
  For all $f \in M_n(\F_q)$ with roots $z_1, \dots, z_n \in \F_{q^n}$, observe that $f$ is a power of an irreducible polynomial if and only if $\Frob_q(z_1, \dots, z_n) = \sigma(z_1, \dots, z_n)$ for some $n$-cycle $\sigma \in S_n$, in which case $f$ has $\Lambda(f)$ distinct roots. In the notation of Theorem \ref{thm:arith-bound},
  \[
  \Lambda(f) = \sum_{\substack{(z_1, \dots, z_n) \in \F_{q^n}^n \\ f = \prod_i (x - z_i)}} \delta_\sigma(z_1, \dots, z_n).
  \]
  So $\Lambda$ is arithmetic of von Mangoldt type with coefficients $c_\sigma = 1$ for a fixed $n$-cycle $\sigma \in S_n$ and $c_\tau = 0$ for all other $\tau \in S_n$. The result now follows immediately from Theorem \ref{thm:arith-bound}.
\end{proof}

As a slight variant, we may instead look at $\Lambda - 1$, which is normalized to have mean zero.
\begin{cor}
  For any integers $m \geq 2$, $n \geq 4$, and $1 \leq h \leq n - 3$, in the limit as $q \to \infty$ (with the restrictions of Remark \ref{rem:char-restriction}),
  \[
  \frac{1}{q^n} \sum_{f \in M_n(\F_q)} \Bigl( \sum_{\substack{g \in \F_q[x] \\ \deg g \leq h}} [\Lambda(f + g) - 1] \Bigr)^m = \bigO(q^{(h + 1) (m - 1)}),
  \]
  with the implicit constant depending on $m$, $n$, and $h$.
\end{cor}
\begin{proof}
  Since $\Lambda$ is arithmetic of von Mangoldt type, by Lemma \ref{lem:Mangoldt-type}, $\Lambda - 1$ is as well, with coefficients $c_\sigma = 1 - \frac{1}{n!}$ for a fixed $n$-cycle $\sigma \in S_n$ and $c_\tau = -\frac{1}{n!}$ for all other $\tau \in S_n$. The result then follows from Theorem \ref{thm:arith-bound}.
\end{proof}

\begin{remark}
  Similarly, by Proposition \ref{prop:m=1}, we obtain the following well-known analogue of the prime number theorem:
  \[
  \sum_{f \in M_n(\F_q)} \Lambda(f) = q^n.
  \]
\end{remark}

\subsection{M\"obius sums}
For $f \in \F_q[x]$, define the M\"obius function $\mu$ by $\mu(f) = 0$ if $f$ is not squarefree, and $\mu(f) = (-1)^k$ if $f$ is squarefree and has $k$ distinct irreducible factors in $\F_q[x]$.

\begin{cor}
  \label{cor:mobius}
  For any integers $m \geq 2$, $n \geq 4$, and $1 \leq h \leq n - 3$, in the limit as $q \to \infty$ (with the restrictions of Remark \ref{rem:char-restriction}),
  \[
  \frac{1}{q^n} \sum_{f \in M_n(\F_q)} \Bigl( \sum_{\substack{g \in \F_q[x] \\ \deg g \leq h}} \mu(f + g) \Bigr)^m = \bigO(q^{(h + 1) (m - 1)}),
  \]
  with the implicit constant depending on $m$, $n$, and $h$.
\end{cor}
\begin{proof}
  For all squarefree $f \in M_n(\F_q)$ with roots $z_1, \dots, z_n \in \F_{q^n}$, $\Frob_q$ acts on $z_1, \dots, z_n$ by a unique permutation $\sigma_f \in S_n$, and $\mu(f)$ is the number of cycles of the permutation. Let $\eps: S_n \to \set{\pm 1}$ be the sign map. Then
  \[
  \mu(f) = (-1)^n \eps(\sigma_f).
  \]
  By abuse of notation, denote $\eps(f) := \eps(\sigma_f)$ for $f$ squarefree and $\eps(f) = 0$ for $f$ not squarefree. Then
  \[
  \eps(f) = \frac{1}{n!} \sum_{\substack{(z_1, \dots, z_n) \in \aff^n \\ f = \prod_i (x - z_i)}} \sum_{\sigma \in S_n} \eps(\sigma) \delta_\sigma(z_1, \dots, z_n).
  \]
  Indeed, if $f$ is not squarefree, then composing with a transposition that permutes two equal roots of $f$ gives a bijection between even permutations for which $\delta_\sigma(z_1, \dots, z_n) = 1$ and odd permutations for which $\delta_\sigma(z_1, \dots, z_n) = 1$. If $f$ is squarefree, then for each $(z_1, \dots, z_n) \in \aff^n$, $\delta_\sigma(z_1, \dots, z_n) = 1$ for exactly one $\sigma$, and $\eps(f) = \eps(\sigma)$.

  Hence, the result follows from Theorem \ref{thm:arith-bound} with $c_\sigma = \frac{1}{n!} (-1)^n \eps(\sigma)$ for $\sigma \in S_n$.
\end{proof}

\begin{remark}
  Similarly, by Proposition \ref{prop:m=1}, we obtain the well-known result that
  \[
  \sum_{f \in M_n(\F_q)} \mu(f) = \sum_{\sigma \in S^n} c_\sigma q^n = \sum_{\sigma \in S_n} \frac{(-1)^n \eps(\sigma)}{n!} q^n = \frac{(-1)^n}{n!} q^n \sum_{\sigma \in S_n} \eps(\sigma) = 0.
  \]
\end{remark}

\section{Variance computations and cohomology}
\label{section:variance-cohom}
Our results on cohomology do not follow \emph{solely} from bounds on moments of arithmetic functions, but such bounds can be combined with our results to deduce more information about cohomology. We would be interested to know of a purely geometric way to deduce this information.

The main analytic input is Theorem \ref{thm:rodgers} \cite[Thm.\ 3.1]{Rodgers} and the version for covariance \cite[Thm.\ 10.1]{Rodgers}. The cohomology group we study in this section is
\[
\mathcal{H} := \Gr_W^{2n - 2} H^{2n - 2}(X_{2, n, h}, \Q_\ell),
\]
which is the highest weight graded piece of the highest degree of cohomology not computed by Corollary \ref{cor:cohom}, and which controls the main term of the variance of arithmetic functions in short intervals (the implied constant on the $\bigO(q^{h + 1})$ term in Theorem \ref{thm:arith-bound}).

Given $\sigma \in S_n$, let 
\[
M_\sigma(\F_q) := \{f \in M_n(\F_q): f \text{ is squarefree of the same cycle type as } \sigma\},
\]
and let $S_\sigma \subset S_n$ be the conjugacy class of $\sigma$. Consider the factorization function $\varphi_\sigma$ as in the proof of Proposition \ref{prop:Mangoldt-factorization}. Recall that for squarefree $f \in M_n(\F_q)$,
\[
\varphi_\sigma(f) = \begin{cases}
  d_\sigma & \text{if $f$ has cycle type $\sigma$}, \\
  0 & \text{otherwise},
\end{cases}
\]
where $d_\sigma = \#S_n/\#S_\sigma$. Decompose $\varphi_\sigma$ as in \cite[2.2]{Rodgers}:
\[
\varphi_\sigma(f) = \sum_{\lambda \vdash n} \hat{\varphi}_\lambda X^\lambda(f) + b(f) = \sum_{\lambda \vdash n} X^\lambda(\sigma) X^\lambda(f) + b(f),
\]
where $b$ vanishes on the squarefree locus. Indeed,
\begin{align*}
  \hat{\varphi}_\lambda &= \lim_{q \to \infty} \frac{1}{q^n} \sum_{f \in M_n(\F_q)} \varphi_\sigma(f) X^\lambda(f) = \lim_{q \to \infty} \frac{\# M_\sigma(\F_q)}{q^n} \cdot \frac{\#S_n}{\#S_\sigma} X^\lambda(\sigma) \\
  &= \frac{\# S_\sigma}{\#S_n} \cdot \frac{\#S_n}{\#S_\sigma} X^\lambda(\sigma) = X^\lambda(\sigma).
\end{align*}

By \cite[Thm.\ 10.1]{Rodgers}, given $\sigma_1, \sigma_2 \in S_n$, the covariance of $\varphi_{\sigma_1}$ and $\varphi_{\sigma_2}$ in short intervals is given by
\[
\Covar_{f \in M_n(\F_q)}\Bigl( \sum_{g \sim_h f} \varphi_{\sigma_1}(g), \sum_{g \sim_h f} \varphi_{\sigma_2}(g) \Bigr) = q^{h + 1} \sum_{\substack{\lambda \vdash n \\ \lambda_1 \leq n - h - 2}} X^\lambda(\sigma_1) X^\lambda(\sigma_2) + \bigO(q^{h + 1/2}).
\]
On the other hand, using Proposition \ref{prop:m=1}, for $i = 1, 2$,
\[
\frac{1}{q^n} \sum_{f \in M_n(\F_q)} \sum_{g \sim_h f} \varphi_{\sigma_i}(g) = \frac{q^{h + 1}}{q^n} \sum_{f \in M_n(\F_q)} \varphi_{\sigma_i} = q^{h + 1},
\]
so
\begin{align*}
  &\Covar_{f \in M_n(\F_q)} \Bigl( \sum_{g \sim_h f} \varphi_{\sigma_1}(g), \sum_{g \sim_h f} \varphi_{\sigma_2}(g) \Bigr) \\
  &= \frac{1}{q^n} \sum_{f \in M_n(\F_q)} \Bigl( \sum_{g \sim_h f} \varphi_{\sigma_1}(g) - q^{h + 1} \Bigr) \Bigl( \sum_{g \sim_h f} \varphi_{\sigma_2}(g) - q^{h + 1} \Bigr) \\
  &= \frac{q^{h + 1}}{q^n} \sum_{\substack{f_1, f_2 \in M_n(\F_q) \\ f_1 \sim_h f_2}} \varphi_{\sigma_1}(f_1) \varphi_{\sigma_2}(f_2) - q^{2h + 2}.
\end{align*}
Putting these together,
\[
\sum_{\substack{f_1, f_2 \in M_n(\F_q) \\ f_1 \sim f_2}} \varphi_{\sigma_1}(f_1) \varphi_{\sigma_2}(f_2) = q^{n + h + 1} + q^n \sum_{\substack{\lambda \vdash n \\ \lambda_1 \leq n - h - 2}} X^\lambda(\sigma_1) X^\lambda(\sigma_2) + \bigO(q^{n - 1/2}).
\]
By the same argument as in the proof of Theorem \ref{thm:arith-bound},
\begin{align*}
  \sum_{\substack{f_1, f_2 \in M_n(\F_q) \\ f_1 \sim f_2}} \varphi_{\sigma_1}(f_1) \varphi_{\sigma_2}(f_2) &= \sum_{(z_{i, j}) \in Y_{2, n, h}} \delta_{\sigma_1}(z_{1, 1}, \dots, z_{1, n}) \delta_{\sigma_2}(z_{2, 1}, \dots, z_{2, n}) \\
  &= \{(z_{i, j}) \in Y_{2, n, h}: \Frob_q(z_{i, j}) = z_{i, \sigma_i(j)} \, \forall i, j\}.
\end{align*}
So by the Grothendieck--Lefschetz trace formula,
\begin{multline*}
  \frac{1}{q - 1} \sum_{r} (-1)^r \Tr(\Frob_q \circ (\sigma_1^{-1}, \sigma_2^{-1}): H^r(X_{2, n, h}, \Q_\ell)) \\
  = q^{n + h + 1} + q^n \sum_{\substack{\lambda \vdash n \\ \lambda_1 \leq n - h - 2}} X^\lambda(\sigma_1) X^\lambda(\sigma_2) + \bigO(q^{n - 1/2}).
\end{multline*}
By Corollary \ref{cor:cohom} and Proposition \ref{prop:sym-action}, the $q^{n + h + 1}$ term comes from $H^r(X_{2, n, h}, \Q_\ell)$ for $r \geq 2n - 1$, so
\[
\Tr(\Frob_q \circ (\sigma_1^{-1}, \sigma_2^{-1}): \mathcal{H}) = q^{n - 1} \sum_{\substack{\lambda \vdash n \\ \lambda_1 \leq n - h - 2}} X^\lambda(\sigma_1) X^\lambda(\sigma_2).
\]
(The error term vanishes because by construction, $\mathcal{H}$ is pure of weight $2n - 2$.) In particular, taking $\sigma_1 = \sigma_2 = \id$,
\[
\Tr(\Frob_q: \mathcal{H}) = q^{n - 1} \sum_{\substack{\lambda \vdash n \\ \lambda_1 \leq n - h - 2}} \dim(V_\lambda)^2,
\]
where $V_\lambda$ is the irreducible representation of $S_n$ corresponding to $\lambda$.

If $t_1 q^{n - 1}, \dots, t_r q^{n - 1}$ are the eigenvalues of $\Frob_q$, then $t_1, \dots, t_r$ are roots of unity. Choosing $N$ so that $t_i^N = 1$ for all $i = 1, \dots, r$, the eigenvalues of $\Frob_{q^N}$ on $\mathcal{H}$ are $(t_i q^{n - 1})^N = (q^N)^{n - 1}$. Furthermore,
\[
\Tr(\Frob_{q^N}: \mathcal{H}) = (q^N)^{n - 1} \sum_{\substack{\lambda \vdash n \\ \lambda_1 \leq n - h - 2}} \dim(V_\lambda)^2,
\]
so
\[
\dim \mathcal{H} = \sum_{\substack{\lambda \vdash n \\ \lambda_1 \leq n - h - 2}} \dim(V_\lambda)^2.
\]
But this means $t_1, \dots, t_r$ must all be equal to $1$ in the first place, so $\Frob_q$ acts by scalar multiplication by $q^{n - 1}$.

Hence,
\[
\Tr((\sigma_1^{-1}, \sigma_2^{-1}): \mathcal{H}) = \sum_{\substack{\lambda \vdash n \\ \lambda_1 \leq n - h - 2}} X^\lambda(\sigma_1) X^\lambda(\sigma_2),
\]
so we can recover the action of $S_n \times S_n$ on $\mathcal{H}$. In summary, we have proved:
\begin{thm}
  Fix integers $n$ and $h$ with $1 \leq h \leq n - 5$. Let
  \[
  \mathcal{H} := \Gr_W^{2n - 2} H^{2n - 2}(X_{2, n, h}, \Q_\ell).
  \]
  Then $\Frob_q$ acts on $\mathcal{H}$ by scalar multiplication by $q^{n - 1}$, and the action of $S_n \times S_n$ on $X_{2, n, h}$ induces an action on $\mathcal{H}$ such that
  \[
  \mathcal{H} \cong \bigoplus_{\substack{\lambda \vdash n \\ \lambda_1 \leq n - h - 2}} V_\lambda \boxtimes V_\lambda
  \]
  as an $S_n \times S_n$-representation.
\end{thm}

This computation raises two natural questions for further research:
\begin{enumerate}[(1)]
  \item Is there a way to compute $\mathcal{H}$ (including the Frobenius and $S_n \times S_n$ actions) more directly from the geometry of $X_{2, n, h}$, rather than deducing it from analytic estimates?
  \item What should the analogous statement be for the highest unknown degree of cohomology of $X_{m, n, h}$ for $m \geq 3$?
\end{enumerate}

\begin{acknowledgements*}
  The authors thank Jordan Ellenberg for advice throughout this research. We also thank Andrei C\u{a}ld\u{a}raru, Evan Dummit, Daniel Erman, Laurentiu Maxim, and Melanie Matchett Wood for several productive discussions. We are especially grateful to Lior Bary-Soroker and Zeev Rudnick for their helpful feedback. This material is based upon work supported by the National Science Foundation under Grant No.\ DMS-1402620 and NSF RTG award DMS-1502553.
\end{acknowledgements*}

\end{document}